\documentclass{daj}
\usepackage{amsmath,amssymb,amsfonts,amsthm,bbm,graphicx,color,tikz}
\usepackage{bbm}

\dajAUTHORdetails{%
  title = {Fixed-Energy Harmonic Functions}, %% please capitalize all significant words
  author = {Aaron Abrams and Richard Kenyon},
  plaintextauthor = {Aaron Abrams and Richard Kenyon},
  keywords = {Dirichlet problem, Dirichlet energy, harmonic function, bipolar orientation, totally real number field, rectangle tiling},
}   %%% END \dajAUTHORdetails

\dajEDITORdetails{%
   year={2017},
   number={18},
   received={24 June 2016},   % received date: example: 7 January 2017
   revised={8 March 2017},    % Optional revised date (you may comment it out)
   published={6 December 2017},  % published date
   doi={10.19086/da.2730},       % XXX = number of paper, e.g. da006 for paper#6
}   %%% END \dajEDITORdetails

\newcommand{\Z}{{\mathbb Z}}
\newcommand{\R}{{\mathbb R}}
\newcommand{\C}{{\mathbb C}}
\newcommand{\Q}{{\mathbb Q}}

\newcommand{\E}{{\mathcal E}}

\newcommand{\G}{{\mathcal G}}
\newcommand{\F}{{\mathcal F}}

\newcommand{\M}{{\mathcal M}}
\newcommand{\eps}{\epsilon}
\newcommand{\be}{\begin{equation}}
\newcommand{\ee}{\end{equation}}
\newcommand{\sgn}{\text{sgn}}
\newcommand{\old}[1]{}

\newtheorem{theorem}{Theorem}

\newtheorem{cor}[theorem]{Corollary}

\newtheorem{conjecture}{Conjecture}
\newcommand{\Gal}{\text{Gal}}
\newcommand{\one}{{\mathbbm{1}}}

\begin{document}

\begin{frontmatter}[classification=text]
%% EDITOR: this will force the keywords to appear right after the Abstract.
%%   If the abstract is too long and would force the keywords off the
%%   front page, please comment out % [classification=text] above
%%   This way the keywords will be floated on the bottom of the first page
%%   even though the Abstract spills over to the next page.

%%% AUTHOR: Title goes here.  This line is optional.  You must use it
%%   if title has footnote attached or requires nontrivial typesetting,
%%   e.g., inclusion of linebreaks to force nice layout.
\title{Fixed-Energy Harmonic Functions} %% please capitalize all significant words

%%% AUTHOR:
%%% List all authors. If you wish, place grant acknowledgements in \thanks.
%%% In brackets include a short tag for each author.
\author[a]{Aaron Abrams\thanks{Research supported by Simons Foundation award 281189}}
\author[r]{Richard Kenyon\thanks{Research supported by NSF grant DMS-1208191 and the Simons Foundation award 327929}}

%%% AUTHOR: Abstract goes here
\begin{abstract}
We study the map from conductances to edge energies for harmonic functions 
on finite graphs with Dirichlet boundary conditions. 
We prove that for any compatible acyclic orientation and choice of energies there is a unique
choice of conductances such that the associated harmonic function realizes those orientations and energies.
We call the associated function \emph{enharmonic}. For rational energies and boundary data
the Galois group of $\Q^{tr}$ (the totally real algebraic numbers) over $\Q$ permutes the
enharmonic functions, acting on the set of compatible acyclic orientations.
A consequence is the non-tileability of certain polygons by rational-area rectangles. 

For planar graphs there is an enharmonic conjugate
function; together these form the real and imaginary parts of a ``fixed energy'' analytic function.
In the planar scaling limit for $\Z^2$ (and the fixed south/west orientation), 
these functions satisfy a nonlinear analog of the Cauchy-Riemann
equations, namely
\begin{eqnarray*}u_xv_y &=& 1\\u_yv_x&=&-1.\end{eqnarray*}
We give an analog of the Riemann mapping theorem for these functions, as well as a variational
approach to finding solutions in both the discrete and continuous settings. 
\end{abstract}
\end{frontmatter}

%%% AUTHOR: body of paper starts here
\section{Overview}

\subsection{The Dirichlet problem}

The classical Dirichlet problem is to find a harmonic function on a domain that takes specified values
on the boundary.  It is the mathematical abstraction of physical equilibrium problems arising in 
electromagnetism, fluid flow, gravitation, and other areas.  Solutions to the Dirichlet problem describe 
physical phenomena from the vibrations of a drum to the dissipation of heat.  Dirichlet himself was 
investigating the stability of the solar system \cite{dirichlet.bio}.
In the century and a half since Dirichlet's work, there have been countless investigations into
various versions, modifications, generalizations, and special cases of the Dirichlet problem.

We study the discrete Dirichlet problem on a graph, and in particular we are interested in the connection between 
the \emph{edge conductances} and the \emph{edge energies} of the resulting solution.
In the classical problem one fixes the conductances, and then minimizing the Dirichlet energy gives 
a unique solution to the Dirichlet problem, namely a harmonic function on the vertices called the \emph{voltage}.
The edge energies are then obtained from this solution. Thus there is a map, soon called $\Psi$,
from edge conductances to edge energies.

Here we reverse the problem and fix the desired energies of the harmonic function, determining
which choices of conductances lead to these energies.  We show (Theorem \ref{J=1}) that in suitable 
coordinates $\Psi$ has a surprisingly simple Jacobian, which makes analysis of the point preimages very clean.
This set of preimages is finite, parameterized by a certain set of acyclic orientations of the graph 
(Theorems \ref{A} and \ref{B}).  Moreover, one can determine the Dirichlet solutions through a simple equation, 
the \emph{enharmonic Laplace equation} (Theorem \ref{maxima}).  Remarkably, the cardinality of the solution set 
is independent of the boundary values, as long as these are distinct (Corollary \ref{numberoforientations}).  

We begin in Section \ref{sec:main} with notation and an overview of our main results showing existence of 
solutions and giving variational descriptions of these solutions.  In Sections \ref{sec:degree}, \ref{sec:galois},
\ref{sec:analytic}, \ref{sec:tilings} we give consequences of our theorems that highlight connections between 
the Dirichlet problem on a graph and other areas of mathematics, such as number theory (Theorem \ref{allfields}), 
discrete analytic functions (Theorem \ref{scalingthm}), and rectangle tilings (Corollaries \ref{isotopy} and \ref{cuberoot}).  
Some proofs are given; longer ones are deferred to later sections.

Several explicit examples of the main theorems are given in Section \ref{sec:example}.  
In Section \ref{sec:jacobian} we study the Jacobian of $\Psi$.  Sections \ref{sec:AB}, \ref{relax}, \ref{fields}
contain proofs. We pose a few questions in Section \ref{questions}.

\subsection{Main results}\label{sec:main}

Let $\G=(V,E)$ be a connected graph and $B\subset V$ a subset of $\ge 2$ vertices, which we refer to as \emph{boundary vertices}.
We assume that each edge of $\G$ lies on a simple path from a vertex of $B$ to another vertex of $B$.
Let $c_e>0$ be a \emph{conductance} for each edge $e$.  It will sometimes be convenient to refer to edges
by their endpoints, e.g., we also write $c_{xy}$ for the same conductance if $e$ is the edge with endpoints $x,y$.

Given a function $u:B\to\R$
the classical Dirichlet problem is to find a function $f$ on $V$, equal to $u$ on $B$ and 
which minimizes the \emph{Dirichlet energy}
$$\E(f) = \sum_{e=xy} c_e(f(x)-f(y))^2,$$
where the sum is over all edges of $\G$. 
There is a unique 
minimizing function $h$; it is the unique function harmonic at each vertex of 
$V\setminus B$, that is satisfying, for $x\not\in B$,
$$0=\Delta h(x):=\sum_{y\sim x}c_e(h(x)-h(y))$$
where the sum is over neighbors $y$ of $x$.
The function $h$ is sometimes referred to as the \emph{potential} (or, in the context of resistor networks, the \emph{voltage}).
The \emph{potential drop} across an edge $e=xy$ is $h(x)-h(y)$.

We define the \emph{energy} of $h$ on an edge $e=xy$ to be $c_e(h(x)-h(y))^2$,
and 
\be\label{Psiu}\Psi=\Psi_{u}:(0,\infty)^E\to[0,\infty)^E\ee 
to be the map from the set of edge conductances to the set of energies of the associated harmonic function $h$.
See an example in \eqref{Psi1} below. 

Our main theorem relates the conductances and edge energies of a harmonic function. (See the definition of compatible orientation below.)

\begin{theorem}[Existence]\label{A}
Let the graph $\G$, boundary $B$, and boundary values $u:B\to\R$ be fixed.
For each compatible orientation $\sigma$ of $\G$ and tuple $\E\in(0,\infty)^E$ there is a unique choice of positive conductances
$\{c_e\}$ on edges so that the associated harmonic function $h$ has $\sgn(dh)=\sigma$ and
energies $\E$.
\end{theorem}

%Note that the theorem applies only when the specified energies are all nonzero.

Let $\F_u$ be the set of real-valued functions on $\G$ taking value $u$ on $B$ and which have no interior extrema,
that is, no local maxima or local minima except on $B$. 
For any choice of positive conductances, a harmonic function with boundary values $u$ on $B$
necessarily lies in $\F_u$, by the maximum principle for harmonic functions.

Suppose that $f\in\F_u$ takes distinct values at the two endpoints of 
each edge of $\G$.  Then $f$ induces an orientation
of the edges of $\G$, from larger values to smaller values. This orientation has the following three
obvious properties:

\begin{enumerate}
\item it is acyclic,
that is, there are no oriented cycles,
\item it has no interior sinks or sources (as these would
correspond to interior extrema of $f$),
\item there are no oriented paths from
lower boundary values to higher (or equal) boundary values. 
\end{enumerate}

Given $u$, we call an orientation satisfying the above properties \emph{compatible (with $u$)},
and let $\Sigma_u$ be the set of orientations compatible with $u$.
It is not hard to see that for each $\sigma\in\Sigma_u$
there is a function $h\in\F_u$ inducing $\sigma$, that is,
with $dh(e) := h(x)-h(y)$ nonzero on every edge $e=xy$,
and having sign $\sgn(dh(e))=\sigma(e)$ on each edge. The set
$\F_u(\sigma)$ of functions $h\in\F_u$ inducing the orientation $\sigma$ is therefore 
an \emph{open polytope} $\F_u(\sigma)$, since it is defined by (strict) linear inequalities.

The most basic setting of Theorem \ref{A} above is the case where the boundary $B$ consists in just two vertices. 
Let $B=\{v_0,v_1\}$ and $u(v_0)=0, u(v_1)=1$. 
Then $\Sigma_u$ consists of acyclic orientations
which have a unique sink, at $v_0$, and a unique source, at $v_1$. These are called \emph{bipolar orientations}.
Their cardinality is the \emph{beta invariant} of the graph $\tilde\G$, in which an edge is added to $\G$ from
$v_0$ to $v_1$ if there is no such edge already present. The beta invariant, also known as the \emph{chromatic invariant}, is
the derivative at $1$ of the chromatic polynomial of $\tilde \G$, or equivalently the coefficient of $x$ (and of $y$) 
in the Tutte polynomial
$T(x,y)$, see \cite{dFMR}.

By Theorem \ref{A} above, if we fix a graph $\G$ with boundary $B$, a function $u:B\to\R$, and on each edge $e$ an energy $\E_e>0$, the
set of functions in $\F_u$ with energies $\{\E_e\}$ is in bijection with the set $\Sigma_u$ of compatible orientations of $\G$.
We call these functions \emph{enharmonic} (short for ``energy-harmonic'') with respect to the energies $\{\E_e\}$.
Note that a function that is enharmonic with respect to certain energies is also harmonic with respect to certain
conductances, and vice versa; the difference in terminology indicates only which data were used to define the function.

\begin{theorem}[Characterization]\label{maxima}
Let the graph $\G$, boundary $B$, and boundary values $u:B\to\R$ be fixed.
Enharmonic functions with energies $\E\in(0,\infty)^E$ are precisely the local maxima on $\F_u$ of the functional
\be\label{Mfuncl}
M(h) = \prod_{e=xy\in E}|h(x)-h(y)|^{\E_e}.\ee  
The enharmonic function $h$ with energies $\E$ and satisfying
$dh=\sigma$ is the unique maximizer of $M(h)$ on the polytope $\F_u(\sigma)$.
An enharmonic function is one which satisfies
the (discrete) \emph{enharmonic Laplace equation} $Lh(x)=0$ at every interior vertex $x$, 
where $L$ is the nonlinear operator
\be\label{EHLap}
Lh(x) = \sum_{y\sim x} \frac{\E_{xy}}{h(x)-h(y)}\ee
and the sum is taken over all vertices $y$ adjacent to $x$.
\end{theorem}

In practice, for given energies $\{\E_{e}\}$, determining the function $h$ and the conductances
$\{c_{e}\}$ is most easily done via the equation (\ref{EHLap}).  We prove Theorem \ref{maxima}
in Section \ref{relax}.

\subsection{The degree of $\Psi$}\label{sec:degree}

The map $\Psi_u$ of (\ref{Psiu}) from conductances to energies is a rational map (i.e. its coordinate functions are ratios
of polynomial functions of the conductances).
To see this, note first that to obtain the harmonic function $f$ as a function of the conductances and boundary data, 
we have to invert the Laplacian which has entries which are linear polynomials in the conductances; the energies are then 
obtained from $\E_e = c_e(f(x)-f(y))^2$.
Thus $\Psi_u$ can be extended to a rational map $$\Psi_u:\C^E\to\C^E.$$
The inverse map $\Psi_u^{-1}$ is algebraic. Remarkably, for positive real energies all inverse images are positive real: 

\begin{theorem}\label{B}The rational map $\Psi_u$ is of degree $|\Sigma_u|$ over $\C$. 
\end{theorem}
By this we mean that there is an open dense set of $\C^E$ such that the cardinality of the preimage of any point 
in this set is $|\Sigma_u|$.  See Section \ref{sec:AB} for a proof.
Since Theorem \ref{A} provides $|\Sigma_u|$ preimages of any point with positive real coordinates, 
such a preimage must consist only of points of $\C^E$  with positive real coordinates: 
$$\Psi_u^{-1}((0,\infty)^E)\subset(0,\infty)^E.$$
An immediate consequence is

\begin{cor}\label{numberoforientations}
The number of compatible orientations $|\Sigma_u|$ does not depend on the values of $u$, as long as $u$ 
takes distinct values at distinct boundary vertices.
\end{cor}

\begin{proof} 
The degree of a rational family of rational maps is constant for almost every parameter value. Thus for almost every choice of 
values of $u$, the degree of $\Psi_u$ is constant.
As $u$ varies among a (full-dimensional) set of functions
taking distinct values, the set of compatible orientations does not change. So the set of $u$'s 
where the degree of $\Psi_u$ is smaller is a subset
of the set $u$'s where at least two boundary values are the same. 
\end{proof}

\subsection{Galois action}\label{sec:galois}

Suppose the given energies and boundary values lie in $\Q$, with positive energies.  Then the enharmonic equation
\be\label{eHeqns}
\sum_{y\sim x} \frac{\E_{xy}}{h(x)-h(y)}=0\ee
supplemented with the boundary equations $h|_B=u$
which define the enharmonic functions, form a rational system of equations
defined over $\Q$. Thus the absolute Galois group permutes the solutions. 
Theorem \ref{B} shows that the values of $f$ generate a totally real number field.
Thus we have

\begin{cor}\label{dessins}
If the values of $u$ and the coordinates of $\E\in(0,\infty)^E$ are rational then the Galois group of 
$\Q^{tr}$ (the totally real algebraic numbers) 
over $\Q$ permutes the enharmonic functions of Theorem \ref{A}. Equivalently, $\Gal(\Q^{tr}/\Q)$
acts on $\Sigma_u$.
\end{cor}

We conjecture that for 3-connected graphs, the action is transitive\footnote{A graph with boundary is $3$-connected if,
upon adding edges connecting every pair of boundary vertices,
it cannot be disconnected by removing two vertices.}:
\begin{conjecture}
If $G$ is $3$-connected the above Galois action is transitive for generic\footnote{By generic we mean on a Zariski open set.} rational energies and boundary data. 
\end{conjecture}

Considering all acyclic orientations $\bigcup\Sigma_u$ as $\G$ and $u$ vary, this is reminiscent 
of \emph{dessins d'enfants}, a family of combinatorial objects on which the absolute 
Galois group $\Gal(\bar\Q/\Q)$ acts.  In the current setup it is the Galois group of the totally real algebraic
numbers which acts. We prove (see Section \ref{fields}) that the action is faithful, i.e., that every totally real number field arises in this way:

\begin{theorem}\label{allfields}
For every totally real number field $K$ there is a tuple $(\G,B,u,\{\E_e\})$ consisting of a graph $\G$, boundary $B$, rational boundary values $u$ and rational edge energies $\E_e>0$ such that 
$K$ is the number field generated by the values of one of the associated enharmonic functions, and all enharmonic functions
generate isomorphic number fields.
\end{theorem}

We don't know if Theorem \ref{allfields}
holds when the size of the boundary is constrained, for example when the boundary consists in only two points.
We show this however for quadratic fields in Section \ref{sec:quadratic}.
\medskip

\subsection{Discrete analytic functions}\label{sec:analytic}

A \emph{circular planar network} \cite{CIM} is a finite embedded planar graph $\G$ 
with conductances $c:E\to\R_{>0}$ and boundary $B$ consisting of a subset of vertices on the outer face. We define the 
\emph{dual network} $\G^*$ as the usual planar dual of the graph obtained from $\G$ by adding disjoint rays from each
boundary vertex of $\G$ to infinity.  Thus $\G^*$ has one vertex in every bounded face of $\G$ and one vertex in the
exterior face between every pair of ``consecutive'' boundary vertices of $\G$; there is a dual edge connecting dual vertices for every primal edge separating the associated faces and there is a dual edge for each of the infinite rays.
Given a harmonic function $f$ on $\G$ with boundary values $u$ on $B$, a function $g$ on $\G^*$ is said to be a \emph{harmonic conjugate} if
for any primal edge $e=xy$ we have $f(y)-f(x)=c_e(g(t)-g(s))$ where where $s,t$ are the faces adjacent to edge $xy$, with $t$ being on the right when the edge is traversed from $x$ to $y$. These equations are the \emph{discrete Cauchy-Riemann equations}.  They imply in fact that $f$ is harmonic on $\G$, and $g$ is harmonic on the dual graph
with reciprocal conductances on the dual edges. In this setting such a pair $f,g$ is also said to form a \emph{discrete analytic function}.

For a circular planar network $(\G,B,u,\{\E_e\})$ with enharmonic function $f$, an \emph{enharmonic conjugate} 
$g$
is a function $g$ on $\G^*$ satisfying
$$(g(t)-g(s))(f(y)-f(x)) = \E_{xy},$$
where $s,t$ are the faces adjacent to edge $xy$, with $t$ being on the right when the edge is traversed from $x$ to $y$. 
Enharmonicity of $f$ implies enharmonicity of $g$ on the dual graph with the same energies. 
Together $f$ and $g$ are harmonic conjugates for the associated conductances.
We call the pair $(f,g)$ the \emph{fixed-energy discrete analytic function} associated to the boundary data.
These functions are used as $x$- and $y$-coordinates in the rectangle tilings discussed in the next section.
\medskip

It is natural to consider scaling limits of enharmonic functions, that is, continuous limits of enharmonic functions on a sequence of graphs
approximating (in an appropriate sense) a region in $\R^n$. 
These scaling limits depend strongly on both the underlying graph structure and choice of orientation.
We consider here only the simplest nontrivial case where the approximating graph is a scaled subgraph of $\Z^2$, with edges
oriented south and west. 
We prove in Theorem \ref{scalingthm} below
that the natural scaling limit of an enharmonic function $f$ and its enharmonic conjugate $g$ for $\G=\eps\Z^2$ satisfy
a nonlinear analog of the Cauchy-Riemann equation:
\begin{eqnarray}\label{CR}f_x g_y &=& -1\\
f_yg_x&=&1.\nonumber
\end{eqnarray}
Moreover $f,g$ both satisfy the ``continuous enharmonic Laplace equation''
$$\frac{f_{xx}}{f_x^2}+\frac{f_{yy}}{f_y^2}=0.$$
We prove (see Section \ref{RiemannMap}) an analog of the Riemann mapping theorem in this context, namely that 
from any region in an appropriate class there is a ``fixed-energy analytic" map to a rectangle.

In \cite{Bal} a different continuous analog of the fixed-energy problem is considered.

\subsection{Tilings and networks}\label{sec:tilings}

In 1940, Brooks, Smith, Stone, and Tutte \cite{BSST}, building on work of Dehn \cite{Dehn}, 
gave a 
correspondence between the Dirichlet problem for circular planar networks (with two boundary vertices)
and tilings of a rectangle by rectangles;
they called this the \emph{Smith diagram} of the network. See Figure \ref{5tiles} below for an example.
Namely, given such a tiling, one can orient the rectangle so it sits on the $x$-axis, and then form a 
network from the tiling according to the following rules.  Place a vertex at height $y$ for each 
maximal horizontal segment in the tiling at height $y$; then, for each rectangular tile add an 
edge to the network connecting the vertices corresponding to the top and bottom of the tile. On this edge place a conductance
which is the aspect ratio (width over height) of the tile. The result is a circular planar
graph which can be interpreted as a resistor network. 

If we now imagine the top and bottom vertices connected by a battery generating a voltage differential
equal to the height of the rectangle, then the voltage function on the vertices (i.e., the $y$-values) is a 
harmonic function that solves the Dirichlet problem on the network with boundary conditions enforced 
at the top and bottom vertex. The rectangle widths are the current flows $c_e dh$, 
and the areas are the edge energies $c_e(dh)^2$.

The BSST correspondence is invertible:  beginning with a planar network with a specified source and sink
on the outer face and a specified total voltage drop, one can solve for the harmonic function and the currents in the wires.  The resulting 
network can then be represented as a rectangle tiling by reversing the process described above.  

%%%
\old{
Here is a summary of the correspondence:
\begin{equation*}
\begin{aligned}
\mbox{$y$-coordinate} & \longleftrightarrow \mbox{voltage}\\
\mbox{height of rectangle} & \longleftrightarrow \mbox{voltage drop across wire}\\
\mbox{width of rectangle} & \longleftrightarrow \mbox{current flowing through wire}\\
\mbox{aspect ratio (width/height)} & \longleftrightarrow \mbox{conductance of wire}\\
\mbox{height/width of rectangle} & \longleftrightarrow \mbox{resistance of wire}\\
\mbox{area of rectangle} & \longleftrightarrow \mbox{energy dissipated across wire, per unit time}
\end{aligned}
\end{equation*}

Note also that the network need not be planar:  one can draw a Smith diagram of a non-planar 
network as a rectangular tiling of a (non-planar) surface.
}
%%%

A corollary of Theorem \ref{A} applies to Smith diagrams.
This was originally proved in a different manner in
\cite{Wimer} and a number of times since, see \cite{Felsner}.

\begin{cor}\label{isotopy}
Given a rectangle tiling of a rectangle there is an isotopic tiling in which the rectangles have
prescribed areas.
\end{cor} 

Here by \emph{isotopic} we mean one is obtained from the other by displacing continuously 
the walls, floors and ceilings of rectangles in the first to form the second, without degenerating any rectangles
along the way (we also need to add a proviso that at a point where four rectangles meet only one of the two possible perturbations is allowed;
either the horizontal segments to its left and right can move independently up and down, or the vertical segments above and below can move independently left and right, but not both.)   See Section \ref{sec:example} for some examples.

Such maps are also called \emph{rectangular cartograms} \cite{KS} and \emph{floorplans} \cite{floorplans}
and are used e.g. to make geographic maps where the area represents
the measure of some quantity associated to each country, such as population.
\medskip

Another corollary of Theorems \ref{A} and \ref{B} applies to the 
tileability of polygons by rectangles of rational areas.
Such a tiling arises from a circular planar network with a boundary vertex for each horizontal segment of the boundary polygon.
Again the Dirichlet boundary values are the $y$-coordinates of the horizontal boundary segments.

\begin{cor}\label{cuberoot}
If a polygon $P$ can be tiled with rectangles of rational areas, then the horizontal edge lengths are in
a totally real extension field of $\Q(v_1,\dots,v_k)$, the field generated by the vertical edge lengths.
\end{cor}

For example the unit-area polygon of Figure \ref{nontileable} cannot be tiled with rectangles of rational areas, because $2^{1/3}$ is not a totally real algebraic number.

\begin{figure}
\center{\includegraphics[width=2.4in]{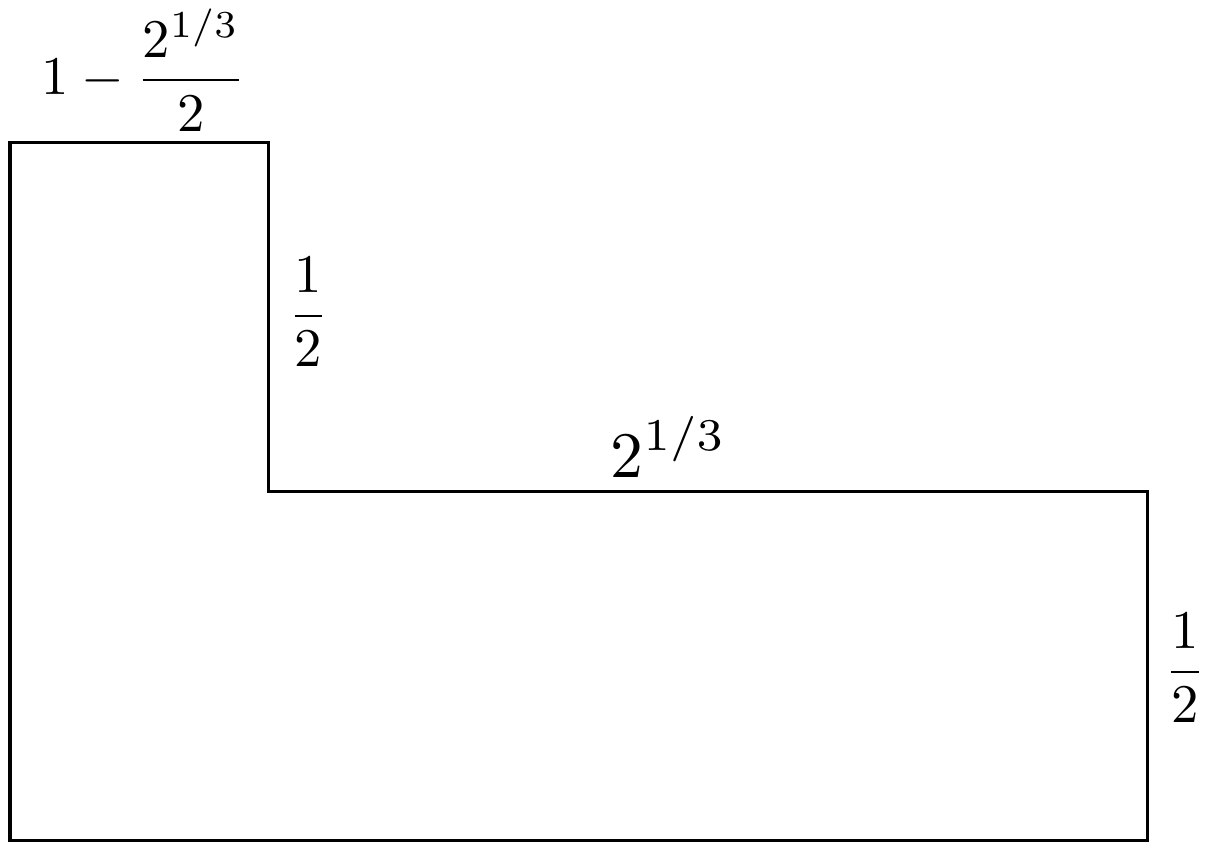}}\vskip-.5in
\caption{\label{nontileable}A polygon with area $1$ which is not tileable with rational-area rectangles.}
\end{figure}

\begin{proof} Assume some vertex of $P$ is at $(0,0)$. As discussed above, the $y$-coordinates 
of the horizontal edges of $P$ correspond to the boundary values $u$ of the enharmonic function. 
The $y$-coordinates of the interior horizontal edges are the interior values of the enharmonic function,
and therefore in a totally real extension of $\Q(v_1,\dots,v_k)$. The $x$-coordinates
of all vertical edges are sums of currents $\E_{xy}/(f(x)-f(y))$, which are therefore in the same extension field.
\end{proof}

\section{Examples}\label{sec:example}

We illustrate our theorems and techniques with a few examples.

\subsection{A small graph}
Consider the graph $\G$ shown in Figure \ref{fig:simple}, with conductances $a,b,c,d,e$ as labeled.
With $u(v_1)=1$ and $u(v_0)=0$, we have $x,y\in[0,1]$, and thus $\F_u\cong[0,1]^2$.  Solving for $h$ gives 
$h(x)=\frac 1{Z} (ab+ac+bc+ae)$ and 
$h(y)=\frac 1{Z}(ab+ac+bc+bd)$, where $Z=ab+ac+ae+bc+bd+cd+ce+de$ is the determinant 
of the Laplacian.  The energies are then 
\be\label{Psi1}
\begin{aligned}
\Psi_u(a,b,c,d,e) = \frac 1{Z^2}\Bigl(&a(bd+cd+ce+de)^2,\ b(ac+cd+ce+de)^2,\\
&c(ae-bd)^2,\ d(ab+ac+bc+ae)^2,\ e(ab+ac+bc+bd)^2 \Bigr)
\end{aligned}
\ee

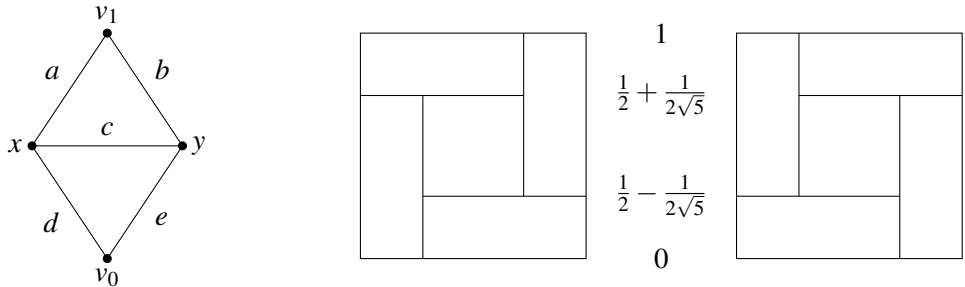
\begin{figure}
\begin{center}
\begin{tikzpicture}
\matrix (m) [row sep = 8em, column sep = 5em]
{
\begin{scope}[scale=.5]
\draw (0,0) -- (2,3) -- (0,6) -- (-2,3) --cycle;
\draw (2,3) to (-2,3);
\draw[fill] (0,0) circle (3pt);
\draw[fill] (0,6) circle (3pt);
\draw[fill] (2,3) circle (3pt);
\draw[fill] (-2,3) circle (3pt);
\draw (0,0) node[anchor=north]{$v_0$};
\draw (0,6) node[anchor=south]{$v_1$};
\draw (2,3) node[anchor=west]{$y$};
\draw (-2,3) node[anchor=east]{$x$};
\draw (-1,4.5) node[anchor=south east]{$a$};
\draw (1,4.5) node[anchor=south west]{$b$};
\draw (0,3) node[anchor=south]{$c$};
\draw (-1,1.5) node[anchor=north east]{$d$};
\draw (1,1.5) node[anchor=north west]{$e$};
\end{scope}
&
\begin{scope}[scale=.5]
\draw (0,0) rectangle (6,6);
\draw (0,4.34) to (4.34,4.34);
\draw (4.34,6) to (4.34,1.66);
\draw (1.66,1.66) to (6,1.66);
\draw (1.66,0) to (1.66,4.34);
	\begin{scope}[xshift=10cm]
	\draw (0,0) rectangle (6,6);
	\draw (1.66,4.34) to (6,4.34);
	\draw (4.34,4.34) to (4.34,0);
	\draw (0,1.66) to (4.34,1.66);
	\draw (1.66,1.66) to (1.66,6);
	\end{scope}
\draw (8,0) node{$0$};
\draw (8,6) node{$1$};
\draw (8,1.66) node{$\frac 12 - \frac 1 {2\sqrt{5}}$};
\draw (8,4.34) node{$\frac 12 + \frac 1 {2\sqrt{5}}$};
\end{scope}
\\};

\end{tikzpicture}
\caption{\label{5tiles}A graph with $|\Sigma|=2$.  With $u(v_0)=0, u(v_1)=1$ and all energies equal, the underlying 
number field is $\Q(\sqrt{5})$.}
\label{fig:simple}
\end{center}
\end{figure}

To find solutions with (for instance) all energies equal to $1$, it is much easier to solve the enharmonic
equation $Lh=0$ than it is to set each coordinate above equal to 1 and solve the system directly.  The enharmonic equation(s)  (see \eqref{eHeqns})
are
\begin{equation*}
\begin{aligned}
0&=\frac {\E_a} {h(x)-h(v_1)} + \frac {\E_c} {h(x)-h(y)} + \frac {\E_d} {h(x)-h(v_0)}=\frac 1 {h(x)-1} + \frac 1 {h(x)-h(y)} + \frac 1 {h(x)}\\
0&=\frac {\E_b} {h(y)-h(v_1)} + \frac {\E_c} {h(y)-h(x)} + \frac {\E_e} {h(y)-h(v_0)}=\frac 1 {h(y)-1} + \frac 1 {h(y)-h(x)} + \frac 1 {h(y)}
\end{aligned}
\end{equation*}
and have the solutions $h(x)=\frac 1 2 \pm \frac {\sqrt{5}}{10}$ and $h(y)=\frac 1 2 \mp \frac {\sqrt{5}}{10}$,
from which one can quickly produce the conductances (for example, $a=e=\frac 12 (15\pm5\sqrt{5})$).
The values of $h$ are roots of the polynomial $5z^2-5z+1$.  There are two solutions because current 
can flow either way along the edge $xy$.  That is, $|\Sigma|=2$.
The two polytopes in $\F_u=[0,1]^2$ on which $M(h)$ is nonzero are the triangles above
and below the diagonal.
 
The equal-area tilings associated to the two solutions are shown in Figure \ref{fig:simple}; the $y$-coordinates of the 
horizontal edges are $h(x)$ and $h(y)$.  The tilings are not isotopic.

\subsection{Quadratic number fields}\label{sec:quadratic}

Keeping the same graph $\G$ as in the previous example, if we instead look for solutions with arbitrary 
energies $\E_a,\E_b,\E_c,\E_d,\E_e$, we find that the conductances
generate the number field $\Q(\sqrt{\delta})$, where 
$$\delta=(\E_a\E_e-\E_b\E_d+\E_cS)^2+4\E_b\E_c\E_dS.$$  
Here $S$ is the sum of the five energies.  
Of course for positive real energies, $\delta>0$ as all roots are real.  It turns out that we can get any real
quadratic number field this way.  To see this, we first specify $\E_a=\E_b=\E_c=1$ so that
the discriminant reduces to $\delta=4\E_d^2+4\E_d\E_e+4\E_e^2+12\E_d+12\E_e+9$.
To obtain the number field $\Q(\sqrt{D})$ for a positive squarefree integer $D$, it suffices to find
positive rational numbers $\E_d,\E_e$ such that $\delta=Dr^2$
for some rational $r$. Choose a rational $s$ so that $D' = D s^2\in(1/3,4/9).$
Then substituting $\E_d = \frac1{6D'-2}, \E_e =\frac{4-9D'}{6D'-2}$ (both positive) gives discriminant
$\delta = \frac{9}{(3D'-1)^2}D'$ which is a rational square times $D$.

%%%
\old{
 A change of variables ($u=\E_d+2\E_e+3, v=\E_d+1$) followed by clearing denominators
produces the equation 
\begin{equation}\label{Pell} w^2-Dz^2=-3(x^2-y^2),\end{equation}
which has many integer solutions.  For instance we can set $w=3, z=6, x=y+1=6D-1$.
This gives energies $\E_d=\frac{x-y}y=\frac 1{6D-2}, \E_e=\frac{w-x-2y}{2y}=\frac{4-9D}{6D-2}$
that are rational but not necessarily positive.   
Note $\E_d>0$, but to get $\E_e>0$ we need to solve \eqref{Pell} with $w-x-2y>0$.

We can do this by viewing \eqref{Pell} as a Pell equation in the variables $w$ and $z$, keeping $x$ and $y$
as already defined.  The equation asserts that the norm of the vector $3+6\sqrt{D}$ in the ring $\Z[\sqrt{D}]$
is $-3(x^2-y^2)=9-36D$.  We can obtain other vectors with the same norm by multiplying this vector by any norm
1 unit in this ring.  Since there are norm 1 units of the form $p+q\sqrt{D}$ for arbitrarily large $p$
and $q$, we can replace $w$ and $z$ with arbitrarily large values\footnote{If $D\equiv 1 \mbox{ mod } 4$ then
$p$ and $q$ may be half-integers, but we don't need $w\in\Z$:  a rational $w$ still leads to rational $\E_d,\E_e$.} 
without changing $x$ and $y$.  This will yield positive rational energies $\E_d,\E_e$.

For instance if $D=11$ then we start with the equation 
$3^2-11\cdot 6^2=-3(65^2-64^2)$ and then replace the left side with the norm of the product of $3+6\sqrt{11}$
with the unit $10+3\sqrt{11}$ of $\Z[\sqrt{11}]$, resulting in the desired solution $w=228, z=69, x=65, y=64$.
This translates to $\E_d=1/64, \E_e=35/128$.
}
%%%

\subsection{Jacobi polynomials}
One way to generalize the previous graph is the following.\footnote{Note that $n=2$ is the only member of this
family that is planar.}  Let $\G=K_{n+2}$ be the complete graph on $n+2$ vertices $v_0,\dots,v_{n+1}$,
minus the edge $v_0v_1$.
Fix boundary $B=\{v_0,v_1\}$ and boundary values $u(v_0)=0$ and $u(v_1)=1$. Put energy $a$ on all edges 
from $v_0$ to $v_i$, $i\in[2,n+1]$ and energy $b$ on edges from $v_1$ to $v_i,~i\in[2,n+1]$,
and energy $2$ on the remaining edges. Then by symmetry we can order the values $h(i)$ for  $i\in[2,n+1]$ in increasing order.
We have
$$M(h) = \prod_{i\in[2,n+1]} h(i)^a(1-h(i))^b\prod_{1<i<j<n+1}(h(j)-h(i))^2.$$
In \cite[Theorem 6.7.1]{Sz} it is shown that
the values $h(i)$ maximizing $M(h)$ are the roots of the Jacobi orthogonal polynomial $P_{a-1,b-1}(z)$, scaled to the interval
$[0,1]$.
\medskip

\subsection{A medium graph}
For another example, consider the network shown in Figure \ref{graph12}.  
Again let $u(v_1)=1$ and $u(v_0)=0$.  The edges are labeled 1 through 12,
and (by coincidence) also $|\Sigma_u|=12$.  
We have solved the enharmonic equation (see \eqref{eHeqns}) for the 
twelve enharmonic functions that produce all energies equal to $1$.  From these functions we computed the 
conductances of the networks, which determine the shapes of the rectangles in the corresponding Smith
diagrams.  In Figure \ref{example12} we have scaled the Smith diagrams to be (unit) squares, so each 
rectangle in each rectangulation has area $1/12$.
Any rectangulation of a rectangle with the same underlying network (Figure \ref{graph12}) 
is isotopic to one of these twelve. 

These calculations are too complicated to do by hand; for instance, according to Mathematica, the width of the 
rectangle labeled ``1'' in each picture is a root of the polynomial
$$
\begin{aligned}
p(z)=\ &1270080000000 z^{12} - 5554584000000 z^{11} + 10776143400000 z^{10} \\
& -  12235337185000 z^9 + 9034493949125 z^8 - 4560532680000 z^7 \\
& +  1610724560815 z^6 - 400501165895 z^5 + 69535433439 z^4  \\
& - 8223166134 z^3 + 629396649 z^2 - 28041714 z + 551124.
\end{aligned}
$$
All other edge lengths in all tilings in Figure \ref{example12} lie in the splitting field of $p$, which is contained in $\R$.

\begin{figure}[htbp]
\center{\includegraphics[width=1.8in]{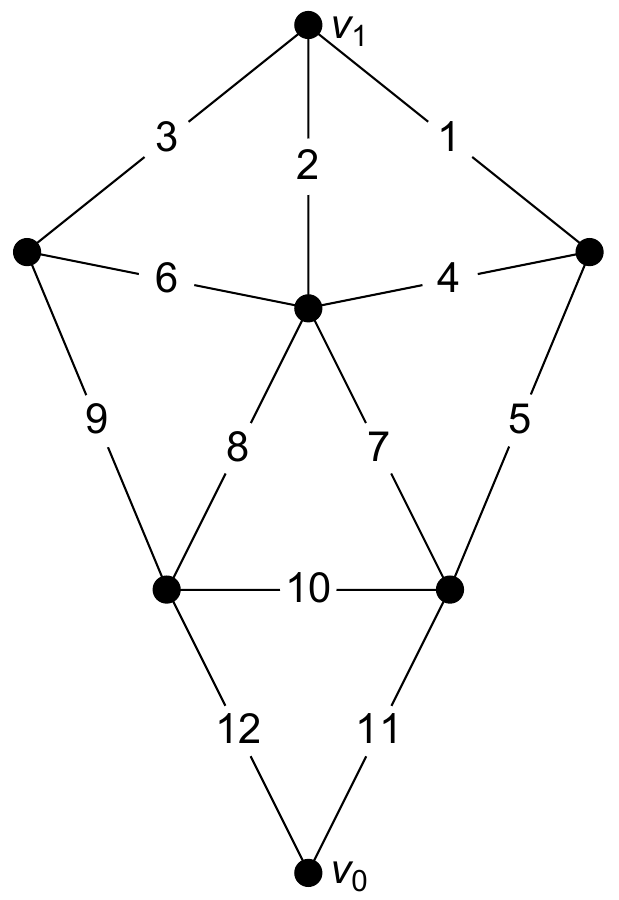}}
\caption{\label{graph12}A planar network with two boundary vertices $v_0,v_1$.}
\end{figure}

\begin{figure}[htbp]
\center{\includegraphics[width=6in]{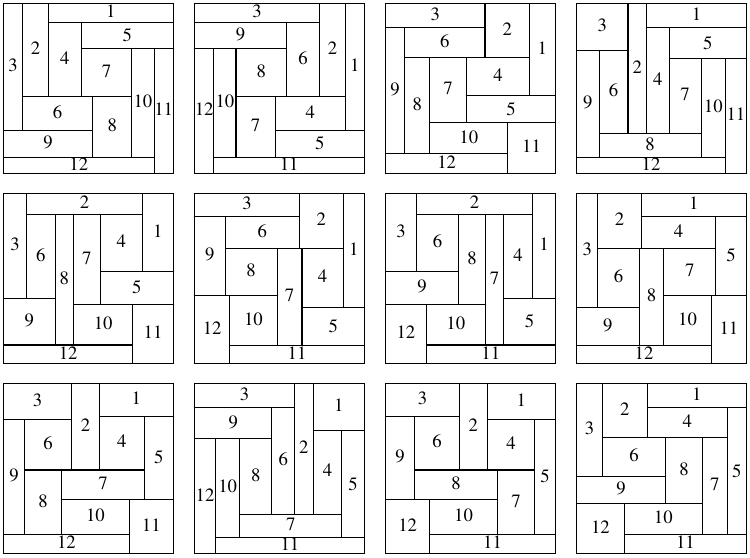}}
\caption{\label{example12}Smith diagrams for the planar network of Figure \protect{\ref{graph12}}, with energies $1$. 
Here $|\Sigma_u|=12$, and the Galois group
acts transitively: the underlying number field for each tiling is a degree-$12$ totally real extension field of $\Q$. Each tiling corresponds to a different bipolar orientation of $\G$, so these tilings are not isotopic.}
\end{figure}

\subsection{A large graph}
The variational method makes it very easy to find (numerically) enharmonic functions on large graphs,
see for example Figure \ref{grid40X40}. In fact this figure illustrates a certain scaling limit convergence theorem, Theorem \ref{scalingthm} below.
\begin{figure}[htbp]
\center{\includegraphics[width=5in]{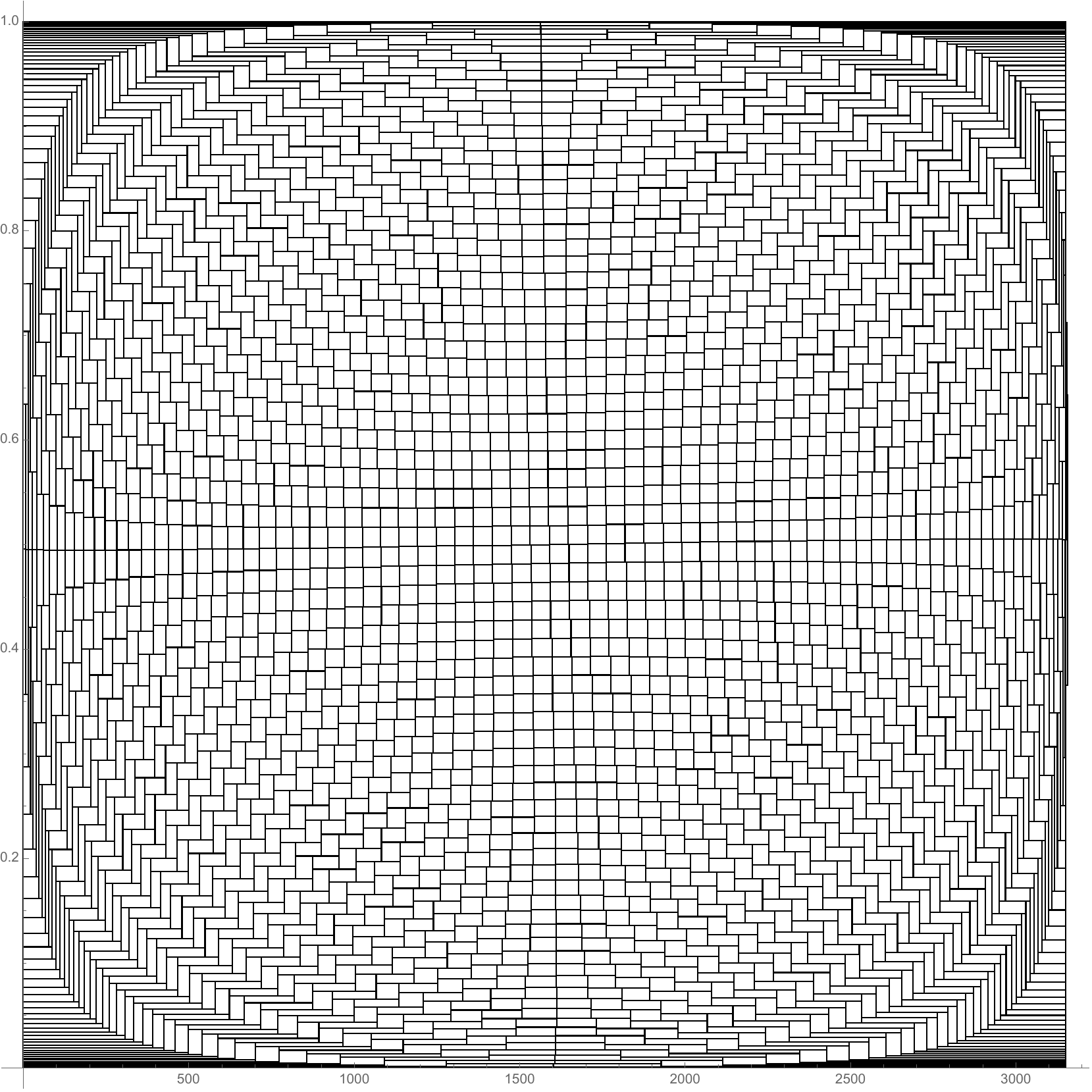}}
\caption{\label{grid40X40}Smith diagram for the $40\times40$ square grid $[0,40]^2\cap \Z^2$, with
edges oriented south and west, with all energies $1$ and boundary values $u(0,0)=0,~u(40,40)=1$. This picture illustrates
Theorem \protect{\ref{scalingthm}} below: the $y$-coordinates of the tops and bottoms of the rectangles, which are indexed by an $n\times n$ grid, 
converge to an analytic function on the square satisfying the enharmonic Laplace equation \eqref{enharmpde}.}
\end{figure}

%XXX
\section{The Jacobian}\label{sec:jacobian}

The work in this paper began with the observation that the Jacobian of $\Psi$ has a 
surprisingly simple description, which we give here. Our proof of Theorem \ref{A} uses the nonvanishing 
of this Jacobian. 

First some definitions.
On a connected graph $\G=(V,E)$ with boundary vertices $B$,
a \emph{$1$-form} is a function $\omega$ on oriented edges which is antisymmetric under changing orientation: 
$\omega(-e)=-\omega(e)$ where $-e$ represents the edge $e$ with reversed orientation. 
Let $\Gamma^1$ be the space of $1$-forms.
Let  $\Gamma^0$ be the space of functions on $V$ with zero boundary values.

We define the map $d:\Gamma^0\to\Gamma^1$ by $df(e) = f(y)-f(x)$ when $e=\vec{xy}$ is directed from $x$ to $y$. 
A $1$-form is a \emph{coboundary} if it is of the form $df$ for some $f\in\Gamma^0$. 
Let $W_{cob}\subset\Gamma^1$ be the space
of coboundaries. It has a natural basis consisting of the set of coboundaries $\{d\one_v\}_{v\in V\setminus B}$ where 
$\one_v$ is the function which is $1$ at $v$ and zero elsewhere.

The orthogonal complement (for the standard inner product on $\R^E$)
of $W_{cob}$ is the space $W_{cyc}$ of \emph{(relative) cycles}, or \emph{flows}.
These are $1$-forms in the kernel of the boundary operator $d^*:\Gamma^1\to\R^V/\R^B\cong \Gamma^0$, 
the transpose of $d$.
The mapping $d^*$ from $1$-forms to functions is
$$d^*\omega(x) = \sum_{y\sim x} \omega(\vec{xy})$$
where the sum is over vertices $y$ neighboring $x$; in other terminology $d^*\omega$ is the \emph{divergence} of $\omega$. Cycles (or flows)
are $1$-forms which are divergence-free on the non-boundary vertices. 

Note that $W_{cob}$ has dimension $|V|-|B|$, the number of internal vertices, 
and $W_{cyc}$ has dimension $|E|-|V|+|B|$.

The differential of the map $\Psi_u$ has a surprisingly simple form when considered as a map 
from the logarithms of the conductances to the logarithms of the energies.
We assume the boundary values $u$ are fixed. If we fix an element $\sigma\in\Sigma_u$, we can identify 
$\R^E$ with the space of $1$-forms.

Let $P_{cyc}:\Gamma^1\to\Gamma^1$ be the orthogonal projection onto $W_{cyc}$, and $P_{cob} = I-P_{cyc}$ the 
projection onto $W_{cob}$. Given a conductance function $c$, let $h$ be the associated harmonic function, 
and $\omega= c\cdot dh$ the associated current flow. We assume $\omega$ is nonzero on each edge,
which holds for almost all choices of $c$.
Let $Q:\R^E\to\R^E$ be the linear map determined by setting $Q\beta = dh\cdot\beta$ for $\beta\in W_{cob}$ and 
$Q\gamma = c dh\cdot\gamma$ for $\gamma\in W_{cyc}$, where the products are pointwise on edges (that is,
$Q\beta(e) = dh(e)\beta(e)$ and similarly for $Q\gamma$).

\begin{theorem}\label{J=1}
When $\Psi_u(\{c\})$ is nonzero on every edge,
let $J_{\log}$ be the differential of the map $\log\circ\Psi_u\circ\exp:\R^E\to\R^E$. 
Then $J_{\log}=Q^{-1}(P_{cyc}-P_{cob})Q$. 
\end{theorem}

\begin{cor}\label{J=1c}
$$\det D\Psi_u =(-1)^{|V|-|B|}\prod_e (h(x)-h(y))^2.$$
In particular on $(0,\infty)^E$, $\Psi_u$ fails to be locally injective precisely when some edge has energy $0$.
\end{cor}

\begin{proof}[Proof of Corollary \ref{J=1c}] 
By the Theorem (and the chain rule)
$$\det D\Psi_u = (-1)^{\text{dim}(W_{cob})}\frac{\prod_e\E_e}{\prod_e c_e}.$$ Substitute $\E_e=c_e(h(x)-h(y))^2$ for edges $e=xy$.
\end{proof}

\begin{proof}[Proof of Theorem \ref{J=1}] 
Let $c, h$ and $\omega$ be as above.
Fix $v\in V\setminus B$. Let $f\in\Gamma^0$ and for $t\in\R$, let $h_t = h + tf$.
Let $\eps>0$ be small enough so that for $t\in(-\eps,\eps)$ the sign of $dh_t$ is constant.
We fix the flow $\omega$ to its initial value, and define the conductance at time $t$ by $c_e(t) = \omega(e)/dh_t(e)$.
Then $h_t$ is harmonic for conductances $c_e(t)$. 
The associated energies are $\E_e(t) = \omega(e) dh_t(e)$ and
thus 
\be\label{Ec}\frac{\partial}{\partial t}\log \E_e(t) = -\frac{\partial}{\partial t}\log c_e(t).\ee

Evaluating at $t=0$, we have
$$\left.-\frac{\partial}{\partial t}\log c_e(t)\right|_{t=0}=\left.\frac{\partial}{\partial t}\log dh_t(e)\right|_{t=0} = \frac{df(e)}{dh(e)}.$$
This quantity $\frac{\partial}{\partial t}\log c_e(t)$ is thus an element of $Q^{-1}W_{cob}$.
Equation (\ref{Ec}) shows that it is in fact an eigenvector of $J_{\log}$ with eigenvalue $-1$.

Now let $\gamma\in W_{cyc}$. We take a different perturbation of the circuit (recycling the notation $t,\eps$).
 Let $\eps>0$ be small enough so that 
$\omega_t := \omega+t\gamma$ has constant signs for all $t\in(-\eps,\eps)$.
We now fix $h$, and define the conductance $c_e(t)$ by $c_e(t) = \omega_t(e)/dh(e)$.
We then have 
\be\label{Ec2}\frac{\partial}{\partial t}\log \E_e(t) = \frac{\partial}{\partial t}\log c_e(t).\ee
Evaluating at $t=0$, this vector is a multiple of $\gamma$:
$$\left.\frac{\partial}{\partial t}\log c_e(t)\right|_{t=0} = \left.\frac{\partial}{\partial t}\log\omega_t(e)\right|_{t=0} = \frac{\gamma(e)}{\omega(e)}.$$
Thus $\frac{\partial}{\partial t}\log c_e(t)$ is an element of $Q^{-1}W_{cyc}$,
and is fixed by $J_{\log}$.

Combining these two results shows that $J_{\log}$ has the desired form.
\end{proof}

It is worth pointing out the interpretation of this proof in terms of tilings.  If $G$ is planar, the proof 
starts with a rectangle tiling with combinatorial structure described by $G$, and then studies
the effect of isotopy on the diagram.  Moving a horizontal segment corresponds to changing the 
value of $h$ at a single point, which is an eigenvector of $J_{\log}$ of eigenvalue $-1$:
the conductances (widths over heights) change in the opposite direction as the areas (widths times heights).  
Moving a vertical segment corresponds to the eigenvalue $+1$: the conductances change in the same direction 
as the areas.

\section{Proofs of Theorems \ref{A} and \ref{B}}\label{sec:AB}

We assume $u$ is fixed and takes distinct values at distinct boundary vertices, and we omit it from the notation.

\begin{proof}[Proof of Theorem \ref{A}]
We first prove that the map is surjective onto $\Sigma$, that is, every $\sigma\in\Sigma$ is realized by some choice of conductances.
Given $\sigma\in\Sigma$, let $f$ be any function in $\F(\sigma).$
Such a function exists because $\sigma$ is acyclic: for example one can define $f$ from a linear extension of $\sigma$, thinking of $\sigma$ as a partial order.
There is a choice of conductances for which this $f$ is harmonic: to see this
take any flow $\omega$ with orientation $\sigma$ and divergence $0$ on $V\setminus B$, and define the conductance
$c_e$ on edge $e=xy$ by $c_e = \frac{\omega(e)}{f(x)-f(y)}$. As a concrete example of such a flow $\omega$, let $P$
be the set of all possible oriented paths from a boundary vertex to a boundary vertex, with orientation compatible with $\sigma$, and let the value 
$\omega(e)$ be the probability that $e$ is on a uniformly chosen path in $P$. By our condition on $\G$ that each edge
lies on a simple path from $B$ to $B$, we see that $\omega$ is nonzero.
This completes the proof that the map is surjective onto $\Sigma$.

Now consider the map $\Psi\colon(0,\infty)^E\to[0,\infty)^E$ sending a tuple of conductances to the corresponding
tuple of energies of the associated harmonic function. Corollary \ref{J=1c} shows that the Jacobian of $\Psi$
is nonzero when the energies are nonzero. 

We claim that $\Psi$ is proper: the preimage of a compact set is compact. In other words, $\Psi$ maps the boundary to the boundary.
Take a sequence of conductance functions leaving every compact set of $(0,\infty)^E$. By taking a subsequence, 
we can assume it converges to a boundary
point of $[0,\infty]^E$. Let $h$ be the associated harmonic function, or potential.
We need to show that along this sequence some edge has energy tending to $0$ or $\infty$.

Let $S_0$ be the set of edges whose conductances are going to zero, and $S_\infty$ those whose conductances are
going to $\infty$. 
Suppose $S_\infty$ does not connect any two boundary vertices. We claim then that $h$
tends to a constant on each component of $S_\infty$. If the potential drop $dh$ across some edge of $S_\infty$
is bounded away from zero, then the Dirichlet energy diverges as this conductance goes to $\infty$.
On the other hand if $h$ is taken to be constant on every component of $S_\infty$ then the Dirichlet energy 
of this $h$ is bounded as the conductances go to $\infty$.
Since the limiting potential minimizes Dirichlet energy, it must be the case that the potential drops $dh$ across edges in
$S_\infty$ tend to zero,  proving the claim.
Likewise the energies on edges in $S_\infty$ tend to zero: suppose some edge of $S_\infty$ 
has energy bounded away from zero. 
By adjusting all values of $h$ by $\eps$ to make them constant on components of $S_\infty$,
we change the energy of edges outside of $S_\infty$ by $O(\eps)$ but decrease the energy on 
$S_\infty$ to be exactly zero.  This is a net decrease in energy.
In particular there is a limiting network obtained by contracting all edges in $S_\infty$.
The resulting contracted network has conductances bounded above, which implies a finite current
flow. The energies on the contracted edges and any edges of zero conductance have gone to zero.

If $S_\infty$ contains a connection between two boundary points (with different values of $u$, by genericity), the current flowing
between these points will diverge. This implies that the energy (which can be written as current flow times $dh$) 
will diverge along some edges of this connection.

Thus $\Psi$ is proper. 

The preimage of $(0,\infty)^E$ in the space of conductances 
is a union of open sets $U_\sigma$ where $dh$ has constant 
sign $\sigma$. At a finite boundary point of any such $U_\sigma$, some energy goes to zero (since zero current implies zero energy).
Hence the boundary of $U_\sigma$ maps into the boundary of the space of energies.
Thus $\Psi$ on any of these open sets $U_\sigma$ is proper and a local homeomorphism (a covering map), and therefore surjective from 
this set to the space of energies.
To get uniqueness it remains to show that $U_\sigma$ is connected.

In fact we show that $U_\sigma$ has the structure of (the interior of) a polytope.
As in the first paragraph of the proof, a point in $U_\sigma$ is determined by
a flow compatible with $\sigma$ (and of divergence zero on interior vertices), and an independent choice of
function $f$ in the open polytope $\F_u(\sigma)$.
The set of current flows $\omega$ compatible with $\sigma$ is the subset of $(0,\infty)^E$ 
defined by the linear equations $\text{div}(\omega)=0$. It is the interior of a convex polytope (possibly
infinite but with finitely many facets).
\end{proof}

%A constructive proof of Theorem \ref{A} is given in section \ref{relax}, below.

\begin{proof}[Proof of Theorem \ref{B}]
We have constructed $|\Sigma|$ real preimages of each point in $(0,\infty)^E$. It remains to show that there are no non-real preimages.
Let $\{c_e\}\in\Psi^{-1}(\{\E_e\})$ be a preimage of a set of positive real energies. Let $h$ be the corresponding solution to the Dirichlet problem with 
conductances $\{c_e\}$. Then for all non-boundary vertices $x$,
$$\sum_{y\sim x}c_e(h(x)-h(y)) = 0$$ and so
$$ \sum_{y\sim x} \frac{\E_{xy}}{h(x)-h(y)}=0.$$ 
We can rewrite this as
$$ \sum_{y\sim x} \frac{\E_{xy}(\overline{h(x)}-\overline{h(y)})}{|h(x)-h(y)|^2}=0,$$
or
$$ \left(\sum_{y\sim x} \frac{\E_{xy}}{|h(x)-h(y)|^2}\right)\overline{h(x)}=\sum_{y\sim x} \frac{\E_{xy}}{|h(x)-h(y)|^2}\overline{h(y)}.$$
This represents $\overline{h(x)}$ as a convex combination of its neighboring values $\overline{h(y)}$, and thus 
$h(x)$ lies in the convex hull of the neighboring values $h(y)$. 
Since the boundary values of $h$ are real, each $h(x)$ must be real as well.
Finally $c_e = \E_e/(h(x)-h(y))^2$ so each $c_e$ is real and positive.
\end{proof}

\section{Enharmonic functions: a variational method}\label{relax}

\subsection{Proof of Theorem \protect{\ref{maxima}}}
\begin{proof}[Proof of Theorem \protect{\ref{maxima}}]
Fix a compatible orientation $\sigma$. Recall the open polytope $\F_u(\sigma)\subset\F_u$ of functions
inducing orientation $\sigma$. Restricted to $\F_u(\sigma)$, 
the function 
$\log M(h)$ is concave, since it is a sum of concave functions 
$$\E(e)\log|h(x)-h(y)| = \E(e)\log(\sigma(e)(h(x)-h(y))).$$
The function $\log M(h)$ tends to $-\infty$ on the boundary of $\F_u(\sigma)$ (when $h(x)=h(y)$ for some edge $e=xy$). 
As a consequence $\log M(h)$, being analytic, is strictly concave on $\F_u(\sigma)$.
It thus has a unique maximizer on the interior of $\F_u(\sigma)$, for each $\sigma$.
The maximizer is determined by setting the derivative with respect to $h(x)$ equal to zero for each vertex $x$, that is
\be\label{EH}\sum_{y\sim x} \frac{\E(e)}{h(x)-h(y)}=0.\ee

Given a solution to \eqref{EH}, set $c_e = \E_e/(h(x)-h(y))^2$, then \eqref{EH} becomes
\be\label{HE}\sum_{y\sim x} c_e(h(x)-h(y))=0,\ee
that is, $h$ is harmonic for the conductances $c_e$ and has energies $\E_e$. 

\end{proof}

\subsection{Scaling limits}
Let $D\subset\R^2$ be a Jordan domain with piecewise smooth boundary and
let $u:\partial D\to\R$ be continuous. We assume 
$u$ is \emph{feasible} in the sense that the space $\Omega(D,u)$ of smooth functions $f\colon D\to\R$
with boundary values $f|_{\partial U}=u$ and satisfying $f_x,f_y>0$ is nonempty\footnote{The issue of feasibility is not completely
straightforward. Such a function $f$ necessarily has level contours having tangents pointing in the NW and SE quadrants. Conversely,
given a foliation $\mathcal F$ of $D$ by SE-directed leaves, in which the value of $u$ on the endpoints of each leaf are the same, there is a function $f$ whose level lines are the leaves of $\mathcal F$.}.
Fix such a function $f_0\in \Omega(D,u)$.

For sufficiently small $\eps>0$ let $D_\eps=\eps\Z^2\cap D$
be the part of the graph $\eps\Z^2$ contained in $D$. 
Let $\sigma$ be the south and west orientation of the edges of $D_\eps$.
Let  $\partial D_\eps$ be the set of boundary vertices of $D_\eps$, that is,
vertices with a nearest neighbor in $\eps\Z^2$ not in $D_\eps$. Let $u_\eps:\partial D_\eps\to\R$ 
be the restriction of $f_0$ to $\partial D_\eps$.

Let $f_\eps$ be the unique enharmonic function on $D_\eps$ with energy $\eps^2$ per edge, 
with boundary values $u_\eps$ and with $df_\eps$ having orientation $\sigma$.

\begin{theorem}\label{scalingthm} 
Under the above hypotheses, as $\eps\to0$ the function $f_\eps$ converges to a continuous function on $D$,
real analytic in the interior of $D$, and satisfying the enharmonic Laplace equation
\be\label{enharmpde}\frac{f_{xx}}{f_x^2} + \frac{f_{yy}}{f_y^2}=0.\ee
There is a unique (up to an additive constant)
conjugate function $g$, also real analytic and continuous up to the boundary,
satisfying
(\ref{CR}).
\end{theorem}

We call $f$ satisfying \eqref{enharmpde} an \emph{enharmonic function}.
A function $g$ satisying (\ref{CR}) is an \emph{enharmonic conjugate}.

%\begin{figure}[htbp]
%\center{\includegraphics[width=5in]{grid40X40}}
%\caption{\label{grid40X40}Smith diagram for the $40\times40$ square grid $[0,40]^2\cap \Z^2$
%with south/west orientation and constant energies. We took $B=\{(0,0),(40,40)\}$.}
%\end{figure}

\begin{proof} Assuming $f\in C^2$ satisfies \eqref{enharmpde},
the $1$-form $\omega=-\frac{dx}{f_y}+\frac{dy}{f_x}$ is closed. 
Let $g(p)=\int^{p}_{p_0}\omega$ for some fixed $p_0\in D$.  
Then $f$ and $g$ satisfy the ``fixed-energy" Cauchy Riemann equations (\ref{CR}).
The enharmonicity of $g$ follows from (\ref{CR}) as well. 

Let $F\in\Omega(D,u)$ be the (unique) function maximizing the concave functional 
\be \M(F)=\int_D\log(F_xF_y)dx\,dy. \ee
Existence and uniqueness of $F$ both follow from strict concavity of $\cal M$ (existence follows from the fact that the space of monotone
functions with boundary values $u$ is compact, and $\M$ is upper semicontinuous; uniqueness follows from the fact that
${\cal M}(\frac12(F_1+F_2))>{\cal M}(F_1)+{\cal M}(F_2)$ unless $F_1=F_2$.
Analyticity follows from analyticity of the function $\log$, see \cite{Morrey}[Theorem 1.10.4 (vi)].

We will show that $f_\eps \to F$.

The Euler-Lagrange equation for a critical point of ${\cal M}(F)$ is the enharmonic equation
$$\frac{F_{xx}}{F_x^2}+\frac{F_{yy}}{F_y^2}=0.$$

Let $M_\eps$ be the functional (\ref{Mfuncl}) for the graph $D_\eps$ with energies $\eps^2$. Let $f_\eps$ be its 
maximizer.  After triangulating the faces of $D_\eps$ by adding diagonals, we can extend $f_\eps$ to a piecewise 
linear function $\tilde f_\eps$ on $D$ (except near the boundary). 
Then $\log M_\eps(f_\eps)$ is a Riemann sum for $\M(\tilde f_\eps)$, so
$$\log M_\eps(f_\eps) =\M(\tilde f_\eps) + o(1).$$

If we restrict $F$ to $D_\eps$ we likewise find by analyticity of $F$ that
$$\log M_\eps(F) = \M(F) + o(1).$$
Consequently
$$\M(F) = \log M_\eps(F)+o(1) \le \log M_\eps(f_\eps)+o(1) = \M(\tilde f_\eps)+o(1) \le \M(F) + o(1).$$
Therefore all these quantities are within $o(1)$ of each other.
By upper semicontinuity of $\M$, we find that $f_\eps$ converges to $F$.

We prove the second statement of the theorem in the coming section.

\subsection{Riemann mapping to a rectangle}\label{RiemannMap}

Let $D\subset\R^2$ be a piecewise smooth domain of the following type. There are four points $a,b,c,d$ on the boundary of $D$ in counter-clockwise order,
so that the boundary of $D$ between $a$ and $b$ is monotone in the NE direction (that is, is monotone in 
the $+x,+y$ directions), 
and similarly between $b$ and $c$, $c$ and $d$, $d$ and $a$ respectively the boundary is monotone in the NW, SW, and SE directions.
Let $u$ be the function which is $1$ on $\partial D$ between $b$ and $c$, and zero between $d$ and $a$. 
Let $D_\eps$ be as before the subgraph of $\eps\Z^2$ with vertices in $D$. Let $a_\eps,\dots,d_\eps$ be boundary vertices
so that the boundary of $D_\eps$ is monotone NE, NW, SW, SE between $a_\eps$ and $b_\eps$, $b_\eps$ and $c_\eps$, $c_\eps$ and $d_\eps$, and $d_\eps$ and $a_\eps$.
Let $D'_\eps$ be the dual graph of $D_\eps$, with a vertex for every bounded face of $D_\eps$ and, instead of one
exterior vertex, four exterior
vertices $v_{ab},v_{bc},v_{cd},v_{da}$ outside of the corresponding boundary edges of $D_\eps$. 
Let $\sigma$ be the south/west orientation of edges of  $D_\eps$.

Let $f\in\Omega(D,u)$ be the unique function on $D$ with boundary values $1$ on $bc$ and $0$ on $da$, free on the other boundaries,
and maximizing $\M(f)$; 
this differs from the previous case because $u$ is only
fixed on part of the boundary, but again existence and uniqueness follow from concavity and analyticity from arguments of \cite{Morrey}.
Then $f$ will satisfy \eqref{enharmpde} on $D$, with fixed boundary values $0$ on $da$ and $1$ on $bc$ respectively and 
along the free
boundary $ab\cup cd$, $f$ will satisfy the condition that $\frac{dx}{f_y}-\frac{dy}{f_x}=0$ 
(this is in fact the Euler-Lagrange condition
along the free boundary). 

As before the function $f_\eps$ on $D_\eps$ maximizing $M$ will lie within $o(1)$ of $f$. 
Using edge energies $\eps^2$, let $g_\eps$ be the enharmonic conjugate of $f_\eps$, defined on the dual 
graph $D'_\eps$ and with value zero on the dual vertex
$v_{ab}$. Let $R_\eps=g_\eps(v_{cd})$ be the value at the opposite dual vertex. 
Since the sum of the energies is the area of the image under the rectangle tiling, 
we have 
$$R_\eps =\eps^2(\text{\# edges of $D_\eps$})+o(1) = 2\text{Area}(D)+o(1).$$
The above convergence argument applied to $g_\eps$ shows that $g_\eps$ converges to $g$, 
the maximizer of $\M$ with boundary values $A=2\text{Area}(D)$ and $0$ on $cd$ and $ab$ respectively. 

It remains to show that $f$ and $g$ are enharmonic conjugates.
The enharmonic conjugate $f^*$ of $f$ is the enharmonic function (well defined up to an additive constant)
which is constant along $ab$ and along $cd$,
with free boundary values along $bc$ and $da$. These are precisely the boundary conditions satisfied by the
function $g$, except that the value of the constant along $cd$ might be different, so we have that 
$f^*= Bg$ for some constant $B$. The map $(x,y)\mapsto (f^*,f)$ is area preserving except for a factor $2$ 
(the Jacobian is identically $2$, by (\ref{CR}))
so the image of $D$ has area $A$, which means the boundary value of $f^*$ 
along $bc$ must also be $A$. Thus $f^*=g$.

This proves that $(f_\eps,g_\eps)$ converges to $(f,g)$. In the limit we have a real-analytic mapping $(g,f)$ 
from $D$ to the rectangle $[0,A]\times[0,1]$.
\end{proof}

\section{Realizing number fields}\label{fields}

\noindent{\emph{Proof of Theorem \ref{allfields}}.}
Let $\G$ be a $d+1$-star graph with vertices $\{v,v_0,\dots,v_d\}$, with $v$ being the central vertex.
Take boundary $B$ consisting of all vertices except $v$. Let $e_i$ be the energy on edge $vv_i$
and take rational boundary values $a_i$ at $v_i$, with $a_0<\dots<a_d$. 
The enharmonic equation for the value $x$ at $v$ is
\be\label{staree}\sum_{i=0}^d\frac{e_i}{x-a_i} = 0.\ee
This equation has exactly $d$ roots, and these are real and interlaced with the $a_i$, that is, there is 
one in each interval
$(a_i,a_{i+1})$ (this follows from Theorem \ref{A} above or simply from the fact that the roots are roots of the derivative
of $\prod(x-a_i)^{e_i}$).

Conversely let $p(x)$ be a polynomial with rational coefficients 
of degree $d$ with all real roots, interlaced with the $a_i$ as above.
Then we claim that 
there is a unique choice of rational $e_i$ so that the enharmonic equation (\ref{staree}) has exactly the same roots.

The map from the tuple of energies $\{e_i\}$ to the tuple of coefficients $\{b_i\}$ of 
the polynomial $p(x) = b_0+b_1x+\dots+b_d x^d$
is a rational linear map from $\R^{d+1}$ to itself.
Specifically its matrix consists of elementary functions in the $a_i$ (after removing $a_{i-1}$ in those for
the $i$th column)
illustrated for the case $d=3$:
$$\begin{pmatrix}b_0\\b_1\\b_2\\b_3\end{pmatrix}
=\begin{pmatrix}
-a_1 a_2 a_3 & -a_0 a_2 a_3 & -a_0 a_1 a_3 & -a_0 a_1 a_2 \\
a_1 a_2+a_3 a_2+a_1 a_3 & a_0 a_2+a_3 a_2+a_0 a_3 & a_0 a_1+a_3 a_1+a_0 a_3 & a_0 a_1+a_2 a_1+a_0 a_2 \\
-a_1-a_2-a_3 & -a_0-a_2-a_3 & -a_0-a_1-a_3 & -a_0-a_1-a_2 \\
1 & 1 & 1 & 1 \\
\end{pmatrix}\begin{pmatrix}e_0\\e_1\\e_2\\e_3\end{pmatrix}
$$

The determinant of this matrix is a polynomial of degree $\binom{d}{2}$ and has factors $a_i-a_j$ for each $i<j$,
so is the Vandermonde $\pm\prod_{i<j}(a_i-a_j)$ and is therefore nonzero. Thus the map from energies to 
coefficients $b_i$ is invertible over $\Q$.
This completes the proof of the claim. 

Now let $K$ be a totally real number field and $x\in K$ a primitive element (i.e. one for which $K=\Q[x]$).
Let $p(X)$ be its minimal polynomial. Let $r_1,\dots,r_d$ be the roots of $p(X)$, which are real.
Replacing $x$ with $x+m$ for a large enough integer $m$ we
can assume that all $r_i>0$. 
Let
$a_1,\dots,a_{d+1}$ be positive rationals interlacing them. 
Now do the above construction with parameters $a_i$.
\hfill$\square$
\medskip

Note that as a simple corollary to this proof, by letting the $e_i$ be arbitrary positive rationals, modulo a global scale,
we can parameterize the polynomials with roots interlaced with the $a_i$ as a simplex
$$\{(e_0,\dots,e_d)~|~\sum e_i = 1\}.$$

\section{Other questions}\label{questions}

\begin{enumerate}
\item
Can one use the volume-preserving character of the map $\log\circ\Psi\circ\exp$ to compute the cardinality 
of $\Sigma$?
This is NP hard in general but integrating the Jacobian over projective space one can approximate the
degree of $\Psi$ and therefore $|\Sigma|$.  
\item
The minimal polynomials for geometric quantities (such as values of the harmonic function) seem to have
coefficients which are fixed-sign integer-coefficient polynomials in the edge energies. Can one give a
combinatorial description of these polynomials?
\item
What can be said about solutions to the enharmonic equation in the case when the boundary values $u$ are complex? The number of solutions is the same,
but there is no longer the same interpretation in terms of orientations.
\end{enumerate}

%\newpage %% AUTHOR: please comment out this line.  It serves only
%%   to demonstrate both types of header line in daj-template.pdf

%\section{Expansion estimates}

% More of the body of your paper goes here~\cite{bergelson-johnson-moreira}.

%%% AUTHOR: optional appendix here
%\appendix %% you may comment this out if no Appendix
%\section*{Appendix}
%%% AUTHOR: optional acknowledgments here
\section*{Acknowledgments} %%  you may comment this out if no Ackno
We thank the anonymous referee for correcting a mistake in an earlier version of Theorem \ref{J=1}.
We also thank Michael Bush, Jeff Lagarias, Russell Lyons, Sanjay Ramassamy, G\"unther Rote, Raman Sanyal, Bernd Sturmfels, and Cynthia Vinzant for helpful suggestions and conversations.

%%% AUTHOR:
%%% Bibliography goes here. Note that the arXiv cannot process bibtex
%%% or biber bibliographies.  Example of acceptable bibliograpy format:
\bibliographystyle{amsplain}

%% AUTHOR: You can generate such a bibliography from a .bib file by 
%% running pdflatex/bibtex/pdflatex/pdflatex and then pasting the .bbl file
%% between \begin{thebibliography} and \end{bibliography}

%%% AUTHOR: Include a short description of each author following the
%%% structure below. Use the same short tags used previously.  
%%% Use \imageat{} and \imagedot{} instead of "@" and "." in
%%% email addresses-this replaces the symbols with graphics to avoid 
%%% e-mail address harvesting from the .pdf file
\begin{dajauthors}
\begin{authorinfo}[a]
  Aaron Abrams\\
  Mathematics Department\\
  Washington and Lee University\\ Lexington VA 24450;\\
  abramsa at wlu.edu \\
  
\end{authorinfo}
\begin{authorinfo}[r]
  Richard Kenyon\\
  Department of Mathematics\\ Brown University\\ Providence, RI 02912\\
  rkenyon at math.brown.edu \\
\end{authorinfo}
\end{dajauthors}

\end{document}